\documentclass[preprint,11pt]{elsarticle}
\usepackage{amsfonts}
\usepackage{eurosym}
\usepackage{amsmath,amssymb,amsthm}
\usepackage{fullpage}
\usepackage[utf8]{inputenc}
\usepackage[T1]{fontenc}

\usepackage{amsmath,amsthm,amsfonts,amssymb}
\usepackage[mathscr]{euscript}
\usepackage{hyperref} 
\usepackage{mathrsfs} 
\usepackage{graphicx}
\usepackage{color} 
\usepackage{epsfig}

\newtheorem{lemma}{Lemma}[section]
\newtheorem{theorem}{Theorem}[section]
\newtheorem{corollary}{Corollary}

\theoremstyle{definition}
\newtheorem{definition}{Definition}
\theoremstyle{definition}
\newtheorem{example}{Example}
\theoremstyle{definition}
\newtheorem{remark}{Remark}
\newtheorem{statement}{Statement}

\numberwithin{equation}{section}

\begin{document}

\begin{frontmatter}
 
\title{On exponential stability of linear and nonlinear delay differential equations: a review and new results}

\author[label1]{Leonid Berezansky}
\author[label2]{Elena Braverman}
\author[label3]{Alexander Domoshnitsky}

\address[label1]{Dept. of Math., Ben-Gurion University of the Negev,
Beer-Sheva 84105, Israel}

\address[label2]{Dept. of Math. and Stats., University of
Calgary,2500 University Drive N.W., Calgary, AB  T2N 1N4, Canada}

\address[label3]{Dept. of Math.,
Ariel University, Ariel, 40700, Israel}

%\date{}
%\maketitle     

%\vspace{ -1\baselineskip}

%{\small
%\begin{center}
 %{\sc Leonid Berezansky}   \\
%%Ben-Gurion University of the Negev,
%Beer-Sheva 84105, Israel 
%\\[10pt]
%{\sc Elena Braverman} \\
%Dept. of Math. and Stats., 
%University of
%Calgary, 2500 University Drive N.W., Calgary, AB,  Canada T2N 1N4
%  e-mail maelena@ucalgary.ca
% \end{center}
%}

%\allowdisplaybreaks

 %\smallskip

 \begin{abstract}
An extensive overview of existing criteria, as well as some new uniform exponential stability tests are included for
a scalar delay equation
$$
\dot{x}(t)+ \sum_{j=1}^n a_j(t)x(h_j(t))=0.
$$
Both cases of continuous and measurable parameters  $h_j$,  $a_j$  are  explored.

We apply the global linearisation approach and employ  linear results 
%These results are applied 
to explore global exponential stability   for nonlinear models of the form
$$
\dot{x}(t)+\sum_{j=1}^n f_j\left( t,x(h_j(t))  \right) =0.
$$
The proofs are based on  solution estimations. 
Further, the Bohl-Perron theorem on exponential dichotomy is instrumental for establishing global  exponential stability for nonlinear models.
Conclusions are illustrated with numerical examples.
\end{abstract}
 

\end{frontmatter}

%{\bf Keywords:}
%Linear and nonlinear delay differential equations, local and global  exponential stability, Bohl-Perron theorem

\section{Introduction}\label{int}
\label{introduct}

The dynamics of applied models in physics, engineering, and population dynamics can frequently be defined by the behaviour of solutions to a scalar differential equation
\begin{equation}\label{1}
\dot{x}(t)+\sum_{j=1}^n f_j\left( t,x(h_j(t)) \right)=0, \quad  t\geq t_0
\end{equation}
with multiple delays (here we assume $h_j(t)\leq t$ for any $j$).

If $f_j(t,0)\equiv 0$,  while the limits $\displaystyle a_j(t) = \lim_{x \to 0} \frac{f_j(t,x)}{x}$  exist and are finite,
the linearised equation about the zero equilibrium for \eqref{1} has the form
\begin{equation}\label{2}
\dot{x}(t)+ \sum_{j=1}^n a_j(t)x(h_j(t))=0, \quad  t\geq t_0.
\end{equation}
Equations \eqref{1} and \eqref{2} are usually considered  with a prescribed initial function
\begin{equation}\label{2A}
x(t)=\psi(t) \mbox{~~ for ~~} t\leq t_0.
\end{equation}

Linearised equations for known  mathematical models have different forms: equations with one, two or several delays,
equations in which all the terms are delayed and those containing a non-delay term, dominating or not,  equations with only nonnegative coefficients or equations
with oscillatory  coefficients, equations with coefficients of prescribed signs, not all of them positive, and equations with bounded, as well as with
 with unbounded delays, inherited from their nonlinear counterparts. 
 
To obtain efficient local  stability conditions for a specific form of nonlinear equation \eqref{1},  a set of explicit, convenient to use
stability tests is required for all the types of linear models listed above.

The first part of the present  paper outlines a review of several known asymptotic and exponential stability tests 
for  \eqref{2} and its special cases. 
We start with a better investigated case of continuous parameters in Section~\ref{sec:continuous}. 
However, in many cases   it is still unknown whether the results for the continuous case can be extended to  measurable
parameters without changing critical constants.
Then, in Section~\ref{sec:other}  we present tests for  measurable parameters. 
We also consider two important classes of equations: with two delays and  when there is a term without delay.

The second part, in particular,  Section~\ref{sec:apriori}, presents new stability conditions for \eqref{2}
in a form quite convenient in applications.
In Section~\ref{applications},  we formulate and apply global linearised stability theorem 
to analyze stability properties of nonlinear models  in the form of \eqref{1}.
In Section~\ref{sec:discussion},  we discuss the reviewed and new results.
We also contemplate on unsolved problems and discuss relevant further research.

We  review here equations with one delay term ($m=1$) only with connection to stability tests for equations with several delays.
Such partial cases  as  autonomous equations with constant coefficients and constant delays, 
equations with periodic parameters, and some other classes of scalar functional differential equations are not considered.
Also,  in the present paper, we don't deal  with  equations of higher order, vector equations and systems of linear equations.

Before we start with the review, let us present basic definitions. Relevant notations are introduced later, in Section~\ref{sec:other}.

\begin{definition}
\label{def1}
We will say that  \eqref{2} is
\begin{itemize}
\item
asymptotically stable (AS) if, in addition to its (local) stability, any solution $x(t)$ %of \eqref{2} 
satisfies $\lim\limits_{t\rightarrow\infty} x(t)=0$;
\item
uniformly exponentially stable (UES) if there are constants ${\mathcal M}>0$ and $\nu >0$  such that any solution 
%$x(t)$ 
of problem \eqref{2}, \eqref{2A} has the estimate
$$
|x(t)|\leq  {\mathcal M} e^{-\nu (t-t_0)} \sup_{t\leq t_0}|\psi(t)|  \mbox{ ~ for ~~} t\geq t_0,
$$
while the numbers ${\mathcal M}$ and $\nu$ are independent  of both the initial point $t_0$ and  the initial function $\psi(t)$.
\end{itemize}
\end{definition}

\begin{definition}
The equilibrium of equation  \eqref{1} (assumed to be zero) is
\begin{itemize}
\item
locally asymptotically stable (LAS) if  for every $\varepsilon>0$ there is a constant $\delta>0$ for which $|\psi(t)|\leq \delta$ implies $|x(t)| <\varepsilon$, $t \geq t_0$ and $\lim\limits_{t\rightarrow\infty} x(t)=0$
 for a solution $x$ to \eqref{1}, \eqref{2A};
\item
locally exponentially stable  (LES)   if  there are  $\delta>0$, ${\mathcal M}>0$ and  $\nu>0$  such that  the estimate 
\begin{equation}\label{2b}
|x(t)| \leq {\mathcal M} e^{-\nu (t-t_0)}  \sup_{t\leq t_0}|\psi(t)| %, ~ N>0, \nu>0,
\end{equation}
holds for all solutions $x$ to \eqref{1}, \eqref{2A}  such that  $|\psi(t)|\leq \delta$;
\item
globally asymptotically stable (GAS) if, in addition to (local) stability, all solution to   \eqref{1}, \eqref{2A}  tend to the equilibrium $\lim\limits_{t\rightarrow\infty} x(t)=0$;
\item
globally exponentially stable (GES) if  \eqref{2b}  is satisfied for any solution to \eqref{1}, \eqref{2A}.
\end{itemize}
\end{definition}

Recall that the fundamental function  $X(t,s)$ of equation \eqref{2} 
solves equation \eqref{2} with the variable initial point $t_0= s$, the zero
initial function $x(t)=0$,  $t<s$, and the initial value equal to one: $x(s)=1$. 
All the stability  definitions for linear equations can be formulated in an equivalent form, using the fundamental function. 
For some equations, such form is more convenient.

\begin{definition}
Linear equation \eqref{2} is called
\begin{itemize}
\item
uniformly stable if the exists ${\mathcal M}>0$ which is an upper bound for $X(t,s)$:  $|X(t,s)|\leq M$ ;
\item
uniformly asymptotically stable, once the limit $\lim\limits_{t\rightarrow\infty}X(t,s)=0$  is uniform in $s$ on $[t_0,\infty)$;
\item 
uniformly exponentially stable if some positive constants ${\mathcal  M} >0$ and $ \nu>0$ define a uniform exponential estimate 
$|X(t,s)|\leq {\mathcal M} e^{-\nu (t-s)}$.
\end{itemize}
\end{definition}

\section{Equations with continuous parameters}
\label{sec:continuous}

In this section, we assume that all the functions $a_j(t)$,  $h_j(t), k=1,\dots, m$ are continuous on $[t_0,\infty)$, 
$\lim\limits_{t\rightarrow\infty} h_j(t)=\infty$,
and the initial function $\psi(t)$ is continuous on $(-\infty, t_0]$.

Equation \eqref{2} with one delay 
\begin{equation}\label{2d}
\dot{x}(t)+a(t)x(h(t))=0
\end{equation}
is very well studied, see   \cite{SYC} for a review of stability results for \eqref{2d}. In particular, the 
following famous stability test \cite{Yoneama} holds for equation \eqref{2d}, see also \cite{Yoneyama1} for unbounded delays.

\begin{theorem}\cite{Yoneama}\label{th20a}
Assume that the coefficient $a(t)$ is nonnegative,
$$
0<\liminf_{t\rightarrow\infty} \int_{h(t)}^t a(s)ds\leq \limsup_{t\rightarrow\infty} \int_{h(t)}^t a(s)ds<\frac{3}{2}.
$$
Then equation \eqref{2d} is AS. The constant $\frac{3}{2}$ is sharp.
\end{theorem}

A fundamental result for equations with continuous parameters and several delays was obtained by Krisztin  \cite{Kz}.

\begin{theorem}\cite{Kz}\label{th20}
1) Let  $0 \leq a_j(t)\leq a_j$, $0\leq t-h_j(t)\leq \tau_j$,  $j=1,\dots,n$,   the integral of the sum of the coefficients diverge
$\displaystyle \int_{\cdot}^{\infty}  \sum_{j=1}^n a_j(s)ds=+\infty$, and
\begin{equation}\label{23a}
\sum_{j=1}^n a_j\tau_j<1.
\end{equation}
Then, \eqref{2} is  AS, and one in the right-hand side of \eqref{23a}
is a sharp constant which cannot be replaced with any $c>1$.

2) If $a_j(t) \equiv  \alpha_j$ are positive constants then one in \eqref{23a} is no longer optimal, a larger constant~$\frac{3}{2}$ in the right-hand side  is  the best possible
for the case.
\end{theorem}

Theorem~\ref{th20} gives explicit stability conditions for equation \eqref{2}. 
However, the important message of this theorem is that, when moving from one to several delay terms, the constant $\frac{3}{2}$ is not preserved and
should be replaced by a smaller best possible constant of $1$. 
Note that for constant coefficients and one term, the best constant is $\pi/2$ exceeding $\frac{3}{2}$.
 
The statement   in \cite[Corollary 2.4]{SYC} adds the following stability test to Theorem~\ref{th20}.

 \begin{theorem}\cite{SYC}\label{th20b}
Let $a_j(t)\geq 0$, $h_n(t)\geq \cdots \geq h_1(t)$. If
 $$
 \limsup_{t\rightarrow\infty}\int_{h_1(t)}^t \sum_{j=1}^n a_j(s)ds<\frac{3}{2} \, ,
 $$
for each solution to \eqref{2}  there is a constant $c$ such that  $\displaystyle  \lim_{t\to\infty} x(t)=c$. 
An additional conditions of divergence
 $\displaystyle \int_{t_0}^{\infty} \sum_{j=1}^n a_j(s)ds=\infty$  implies that  \eqref{2} is AS.
 \end{theorem}
 
Faria in the paper \cite{Faria2022}  investigated a very general linear equation with a non-delay term
which includes equations \eqref{2} with a non-delay term and several delay terms, and equations with a non-delay term and distributed delays. 

Equations \eqref{2} with two terms ($m=2$) often appear as linearised equations of population dynamics.
Therefore, such equations require a more detailed study.

Consider an equation with a non-delay term 
\begin{equation}\label{24a}
\dot{x}(t)=-a(t)x(t)+b(t)x(h(t))
\end{equation}
with $a(t)>0$. 

The following stability result was obtained in \cite{GH} and also follows  from  \cite[Corollary 3.2]{Faria2022}.

\begin{theorem}\cite{Faria2022,GH}\label{th21a}
Assume that there are a constant $T\geq 0$  and a function $e(t)\geq 0$ which is locally integrable, but 
 $\int_T^{\infty} e(t)dt=\infty$. If either \vspace{2mm}
\\
(i) $ \displaystyle   a(t)-|b(t)|\sup_{t\geq T}\int_{h(t)}^t e(u)du \geq e(t)$   \vspace{2mm}
\\
or \vspace{2mm}
\\
(ii) $ \displaystyle   a(t)\geq \alpha |b(t)|$ for some $\alpha>1$    \vspace{2mm}
\\
or  \vspace{2mm}
\\
(iii) 
$ \displaystyle   
e(t)\leq a(t)-|b(t)|,~ \sup_{t\geq T}\int_{h(t)}^t e(u)du<\infty
$ \vspace{2mm}
\\
then  \eqref{24a} is AS.
\end{theorem}

Condition  (ii) implies AS in the case of measurable parameters, 
which is a scalar case of the result, obtained in \cite{BTT}.
Note also that the coefficients in equation \eqref{24a} are not assumed to be bounded for $t\geq T$.

Consider a scalar differential equation including a term with a constant positive coefficient without delay
\begin{equation}\label{24}
\dot{x}(t)+\alpha x(t)+\beta(t)x(h(t))=0.
\end{equation}

  The following result is a corollary of a global asymptotic stability test obtained in \cite{Liz} for a nonlinear equation.

\begin{theorem}\cite{Liz}\label{th21}
Assume that $\alpha>0$,  the delay and the variable coefficients are bounded $0\leq \beta(t)\leq b$,  $0\leq t-h(t)\leq h$, where $b > 0$, $h \geq 0$, and 
\begin{equation*} %\label{25}
\frac{\alpha}{b}e^{- \alpha h}>\ln \frac{b^2+\alpha b}{b^2+\alpha^2} \, .
\end{equation*}
Then equation \eqref{24} is AS. 
\end{theorem}

In \cite{Liz} and several other papers by  Ivanov, Liz, Trofimchuk, Theorem~\ref{th21} was applied to obtain new asymptotic properties for many mathematical models of population dynamics. 
In \cite{BB2015}, Theorem~\ref{th21} was generalized to an equation with several delay terms and measurable parameters.
This result will be formulated in the next section.

Next, proceed to the study of a model containing a non-delay term and one delay term
\begin{equation}\label{27b}
\dot{x}(t)+\alpha(t)x(t)+\beta(t)x(h(t))=0 %, t\geq 0,
\end{equation}
with $\beta(t)\geq 0$.

The test of \cite[Theorem 2.1]{Zhang} gives sharp AS conditions.

\begin{theorem}\cite{Zhang}\label{th24a}
Let
$$
 \liminf_{t\rightarrow\infty}\int_0^t a(s)ds>-\infty,~
 \int_0^{\infty} e^{-\int_s^t a(u)du}b(s)ds\leq \alpha<1.
 $$
 Then the condition $\displaystyle  \lim_{t\rightarrow\infty} \int_0^t a(s)ds=\infty$ is necessary and sufficient for \eqref{27b}  to be AS.
 \end{theorem}
 
Sometimes, linearised equations for  mathematical models have three terms, for example,
there is a production  term, a mortality term, and a harvesting term, where all three can be delayed. The following result 
is designed for this kind of linear equations.
 
Consider an equation with three bounded delays
\begin{equation}\label{27}
\dot{x}(t)+a(t)x(h_1(t))-b(t)x(h_2(t))-c(t)x(h_3(t))=0,
\end{equation}
in particular, $0\leq t-h_j(t)\leq \tau$, $ \displaystyle \int_0^{\infty}a(s)ds=\infty$, the first term 
has a positive coefficient $a(t)  >   0$  and in some sense dominates over the other two.

\begin{theorem}\cite[Theorem 7]{Tan}\label{th23}
Let 
%the first coefficient be nonnegative $a(t)\geq 0$,  and 
for some  $\theta \in (0,1)$,  
$$
|b(t)|+|c(t)|<\theta a(t),~ \limsup_{t\rightarrow\infty} \int_{h_1(t)}^t (a(s)+|b(s)|+|c(s)|)ds<1-\theta.
$$
Then equation \eqref{27} is AS. 
\end{theorem}

 In \cite{GH2001},   the authors considered equation \eqref{2} 
with constant coefficients $a_j(t)\equiv a_i$ and piecewise continuous delays $h_j(t)$.
Denote $$a=\sum_{j=1}^n a_j,  ~\Phi(\tau)=\int_0^{\infty} X(t,\tau)dt,$$ where $X(t,\tau)$  is the fundamental function of the
equation $\dot{x}(t)+x(t-\tau)=0, ~0<\tau<\frac{\pi}{2}$. If $\tau\leq \frac{1}{e}$ then $\Phi(\tau)=1$.
For $\tau\in (\frac{1}{e}, \frac{\pi}{2})$, numerical estimations for $\Phi(\tau)$   are  presented
in \cite[Table 1]{GH2001}.

\begin{theorem}\cite[Theorem 3.2]{GH2001}  \label{th23a}
Suppose $a_i>0$ are constants, %$\sigma_i(t)\geq 0$ 
and there exists a number $\tau  \in [0,\frac{\pi}{2a}]$ such that
$$
\tau a-\frac{1}{\Phi(\tau a)}< \sum_{j=1}^n a_j\liminf_{s \rightarrow\infty}(s-h_j(s))
\leq \sum_{j=1}^n a_j\limsup_{s \rightarrow\infty}(s-h_j(s)) <\tau a+\frac{1}{\Phi(\tau a)}  \, .
$$
Then equation \eqref{2} with constant coefficients is AS.
\end{theorem}

In \cite{TangZou}, the equation with a non-delay term 
\begin{equation}\label{27d}
\dot{x}(t)=-a x(t)+b(t)x(t-\tau)
\end{equation}
(together with more general nonlinear equations) was considered.
As a corollary of  \cite[Theorem~2.5]{TangZou}, we outline a test for \eqref{27d}.

\begin{theorem}\cite{TangZou}\label{th23c}
Let $a>0$, $0\leq b(t)\leq \beta$ and for some $D>0$
$$
\int_{t-\tau}^t e^{-a(t-s)}b(s)ds \leq D<1+\frac{1}{2}\left(1+\frac{a^2}{\beta^2}\right)e^{-a\tau}.
$$
Then equation \eqref{27d} is UAS.
\end{theorem}

\section{Stability tests for equations with measurable parameters}
\label{sec:other}

Next, we consider more general models, where both coefficients and delays are not necessarily continuous.

Up to the end of the paper, we assume that all the parameters of equation \eqref{2} (the coefficients and the delays $t-h_j (t) \geq 0$) are Lebesgue measurable
and essentially bounded on $[t_0,\infty)$, apart from the initial function $\psi(t)$ which is Borel measurable and bounded.
The same hypotheses are implied on all the special forms of  \eqref{2}. 
%the following conditions hold:
%
%(a1) all $a_j(t)$ are Lebesgue measurable essentially bounded on $[t_0,\infty)$ functions;
%
%(a2) all $h_j(t)$ are delayed arguments, $h_j(t)\leq t$, and the delays are bounded $\limsup_{t\rightarrow\infty} (t-h_j(t))<\infty$;
%
%(a3) $\psi(t)$ is Borel measurable and bounded on $(-\infty,t_0]$.

We consider measurable functions (coefficients, delayed arguments, solution derivatives) in the space ${\mathbf L}_{\infty}[t_0,\infty)$ of Lebesgue measurable functions $y: [t_0,\infty) \to {\mathbb R}$ with the essential supremum norm $\| y \|={\rm ess}\sup_{s\geq t_0} |y(s)|$. We will also consider the space ${\mathbf L}_{\infty}[a,b]$ of   $y: [a,b]\to {\mathbb R}$ with  $\| y \|_{[a,b]}={\rm ess}\sup_{s \in [a,b]} |x(s)|$.

Note that initial value problem  \eqref{2}, \eqref{2A} has a unique solution, 
see \cite{AS, BB2006a, Gil3, Hale}.

Together with  homogeneous equation \eqref{2}, we consider its non-homogeneous counterpart with the zero initial function
\begin{equation}\label{5}
\dot{x}(t)+ \sum_{j=1}^n a_j(t)x(h_j(t))=f(t), \quad t\geq t_0, \quad x(t)=0, \quad t\leq t_0,
\end{equation}
where $\|f\|<\infty$, and the norm is defined above. %={\rm ess}\sup_{t\geq t_0} |f(t)|<\infty$. 

For the proofs of all the theorems in this section,  the Bohl-Perron result (Lemma~\ref{lemma1}) is the main tool.
 
\begin{lemma}\cite{AS,Gil3}\label{lemma1}
If for any bounded $f$,  %such that $\|f\|<\infty$,  
a  solution to  \eqref{5} is also a bounded function on  $[t_0,\infty)$.
then \eqref{2} is UES.
\end{lemma}

There is a close connection between the nonoscillatory behaviour and UES for functional differential equations. 
For equation \eqref{2} with nonnegative coefficients, nonoscillation which is interpreted as existence of an  eventually positive solution,
is equivalent to positivity of the fundamental function to  \eqref{2}.
Denote the sum of the coefficients as
\begin{equation}
\label{A_definition}
A(t)=\sum_{j=1}^n a_j(t).
\end{equation}

\begin{lemma}\cite[P. 52]{BBbook}\label{lemma2}
If $a_j(t)\geq 0$ and  the sum of the coefficients defined in \eqref{A_definition} is separated from zero $A(t)\geq a_0>0$ then positivity of the fundamental function to
%the condition $X(t,s)>0, t\geq s\geq t_0$ for the fundamental function $X(t,s)$ of  
\eqref{2}  implies that equation \eqref{2} is UES.
\end{lemma}

Condition $A(t)\geq a_0>0$ can be replaced by a more general condition: 
\\
there is ${\mathcal R}>0$ such that $\displaystyle \liminf\limits_{t\rightarrow\infty} \int_t^{t+{\mathcal R}} A(\zeta)d\zeta>0$, 
\\
which is equivalent to UES of the
ordinary differential equation
$$
\dot{x}(t)+A(t)x(t)=0.
$$

Let us outline some explicit conditions under which  the fundamental function of \eqref{2} is positive.
Many other results can be found in the monograph \cite{BBbook}.

\begin{lemma}\cite{BB2006b, BB2007}\label{lemma34}
Suppose $0\leq a_j(t)\leq \alpha_j$, $t-h_j(t)\leq \tau_j$, $A$ is defined in \eqref{A_definition}, and either
\vspace{2mm} \\
1) $\displaystyle \int_{\min_j \{h_j(t)\}}^t A(s)ds\leq \frac{1}{e}$
\vspace{2mm} \\
or  \vspace{2mm}
\\
2) there exists $\lambda>0$ for which the inequality  $\displaystyle \lambda\geq \sum_{j=1}^n \alpha_j e^{\lambda \tau_j}$ holds.

Then, \eqref{2} possesses a positive solution; moreover, it has a positive  fundamental function. % $ X(t,s)$ of \eqref{2}  is positive.
\end{lemma}

All the stability tests in this section are divided into three groups: 
\begin{enumerate}
\item
equations like \eqref{2}  where the number of delays is not prescribed, 
\item
models  with two terms, both  delayed, 
\item
equations with at least two terms, where one of the terms is not delayed.
\end{enumerate}

\subsection{Equations with multiple delays} 

We start with several explicit sufficient conditions when \eqref{2} is AS or UES.
%The next theorem is  \cite[Theorem 3.1]{BB2011}.

\begin{theorem}\cite[Theorem 3.1]{BB2011}\label{th11}
Let  $a_j(t)\geq 0$ and ~
$\displaystyle 
\limsup_{s \rightarrow\infty}\sum_{j=1}^n \frac{a_j(s)}{A(s)}\int_{h_j(s)}^s  A(\zeta)d \zeta <1+\frac{1}{e} \,
$.

If $\displaystyle  \int_{t_0}^{\infty}  A(\zeta)d\zeta =\infty$ then equation \eqref{2} is AS.

If there is ${\mathcal R} >0$ for which  $\displaystyle \liminf_{t\rightarrow\infty} \int_t^{t+{\mathcal R}}A(\zeta)d\zeta>0$ then \eqref{2} is UES.
\end{theorem}

\begin{corollary}\label{cor3a}
Assume  that $a_j(t)\geq 0$, $A(t)  \equiv \alpha>0$ and %$\displaystyle \sum_{j=1}^n a_j(t)\equiv \alpha>0$ and
\begin{equation}\label{17a}
\limsup_{t\rightarrow\infty}\sum_{j=1}^n a_j(t)(t-h_j(t))< 1+\frac{1}{e}  \, .
\end{equation}
Then equation \eqref{2} is UES.
\end{corollary}

\begin{corollary}\label{cor3b}
Assume that $a_j(t)\geq 0$, while $\displaystyle 0<a_0\leq A(t)\leq A$ and
either
\begin{equation}\label{17b}
\limsup_{s \rightarrow\infty} \int_{\min_j \{h_j(s)\}}^s  A(\zeta) d\zeta <1+\frac{1}{e}
\end{equation}
or
\begin{equation}\label{17c}
\limsup_{s \rightarrow\infty}\sum_{j=1}^n a_j(s)(s-h_j(s))< \frac{a_0}{A}\left(1+\frac{1}{e}\right).
\end{equation}
Then equation \eqref{2} is UES.
\end{corollary}

Let us compare Corollaries~\ref{cor3a} and \ref{cor3b} with Theorem \ref{th20}. For constant coefficients in the statement of 
 Theorem \ref{th20}, the constant $\frac{3}{2}$ is better than  $1+\frac{1}{e} \approx 1.368< 1.5$ in %the right-hand side of inequality
 \eqref{17a}, but the condition $\displaystyle \sum_{j=1}^n a_j(t)\equiv \alpha>0$ is more general than the assumption that each
coefficient is constant. 

Consider the general case of the coefficients in equation \eqref{2}. 
If $\frac{a_0}{A}\left(1+\frac{1}{e}\right)>1$ for the case of variable delays then inequality \eqref{17c} in Corollary \ref{cor3b} is
better than the corresponding inequality in Theorem \ref{th20} with the constant 1. 
Otherwise,  Theorem \ref{th20} is sharper.

In the following theorem \cite[Theorem 6]{BB2006b}  and some other stability tests, 
we do not assume that $a_j(t)\geq 0$. It means that all the coefficients may be oscillatory.
\begin{theorem}\cite{BB2006b}\label{th3}
Let $A(t)\geq \alpha>0$ for $A$ as in \eqref{A_definition}, and there is an argument function $r(t)\leq t$ such that
$$
\int_{r(t)}^t A(s)ds\leq \frac{1}{e}, \quad  \limsup_{t\rightarrow\infty}
\sum_{j=1}^n \frac{|a_j(t)|}{A(t)}\left|\int_{h_j(t)}^{r(t)}\sum_{j=1}^n  |a_j(s)|ds\right|<1.
$$
Then equation \eqref{2} is UES.
\end{theorem}

\begin{corollary}\label{cor5}
Let $A(t)\geq \alpha>0$ and
$$
\limsup_{s \rightarrow\infty}
\sum_{j=1}^n \frac{|a_j(s)|}{A(s)}\left|\int_{h_j(s)}^s \sum_{j=1}^n |a_j(\zeta)|d\zeta \right|<1.
$$
Then \eqref{2} is UES.
\end{corollary}

Next, proceed to the autonomous case
\begin{equation}\label{12b}
\dot{x}(t)+\sum_{j=1}^n a_j x(t-\sigma_j)=0, t\geq t_0,~ \sigma_j\geq 0.
\end{equation}

\begin{corollary}\label{cor5a}
Assume that $\displaystyle \sum_{j=1}^n a_j>0$ and
$$
\sum_{j=1}^n |a_j|\sigma_j\leq \frac{\sum_{j=1}^n a_j}{\sum_{j=1}^n  |a_i|} \, .
$$
Then equation \eqref{12b} is UES.
\end{corollary}

\begin{theorem}\cite{BB2009a}\label{th6}
Let  $A$ as in \eqref{A_definition} be separated from zero $A(t)\geq \alpha>0$. Once there exists a positive solution to
$$
\dot{x}(t)+\sum_{j=1}^n a_j^{+}(t)x(h_j(t))=0,~\mbox{ where ~~}  u^{+}=\max\{u,0\},
$$
%have a positive solution 
(equivalently, its fundamental function is positive),
%. Then equation 
\eqref{2} is UES.
\end{theorem}

For the next  two theorems, in  \eqref{2} %we don't assume that 
the sum of the coefficients 
is not necessarily positive: the inequality $A(t)\geq 0$ can be violated. It means that the sum of the coefficients may be oscillatory.

\begin{theorem}\cite{BB2021}\label{th13}
Denote by $A_0$ the sum of the norms $A_0=\sum_{j=1}^n \|a_j\|$.
% and assume that the equation
Let
\begin{equation}\label{13b}
\dot{x}(t)+A(t)x(t)=0
\end{equation}
be  UES, we recall that the delays are bounded $ 0\leq t-h_j(t)\leq \tau_j$. If 
$$
A_0\int_{t_0}^t   e^{-\int_s^t A(\xi)d\xi}\sum_{j=1}^n \tau_j |a_j(s)| ds\leq \beta<1, \quad  t\geq t_0
$$
then  \eqref{2} is UES.
\end{theorem}

Let us note that equation \eqref{13b} is UES if for some ${\mathcal R}>0$ we have
$$
\liminf_{s\geq t_0} \int_s^{s+{\mathcal R}} A(\zeta)d\zeta>0.
$$

Next, let us present the sum of the coefficients as a positive term $\tilde{a}$ perturbed by a smaller oscillatory part.

\begin{theorem}\cite{BB2021}\label{th14}
Let  $\displaystyle \sum_{j=1}^n |a_j(t)|\leq A_0$,  $s-h_j(s)\leq \tau_j$ and  $A(t)=\tilde{a}(t)+\alpha(t)$, where
$$
\tilde{a}(t)\geq a_0>0, \quad \sup_{t\geq s\geq t_0}\left|\int_s^t \alpha(\xi)d\xi\right|\leq  \alpha_0<\infty, \quad
A_0\sum_{j=1}^n \tau_j  \left\|\frac{a_j}{\tilde{a}}\right\|<e^{-\alpha_0}.
$$
Then equation \eqref{2} is UES.
\end{theorem}

Now, we explicitly assume that the equation has  both positive coefficients and negative ones, and the sign is prescribed for each term. 

\begin{theorem}\cite{BB2009a}\label{th8}
Once  the equation
%Assume that the fundamental function of 
\begin{equation}\label{16}
\dot{x}(t)+ \left. \left. \sum_{j=1}^n  \right[  a_j(t)x(h_j(t))-b_j(t)x(g_j(t))  \right]  =0
\end{equation}
has a positive fundamental function and 
$$
b_j(t)\geq 0, \quad  \liminf_{t\rightarrow\infty} \left.\left. \sum_{j=1}^n  \right[ a_j(t)-b_j(t) \right] >0,
$$$$
h_j(t)\leq g_j(t),  \quad \limsup_{t\rightarrow\infty}(t-h_j(t))\leq H, \quad
H\limsup_{t\rightarrow\infty}b_j(t)<1, \quad j=1,\dots n,
$$
%Then equation 
we have that \eqref{16} is UES.
\end{theorem}

\begin{remark}
If %the fundamental function of equation 
\eqref{2}  has a positive fundamental function
%is positive 
and $b_j(t)\geq 0$, the same is true for \eqref{16}.
%then the fundamental function of \eqref{16} is also positive.
\end{remark}

Consider  the equation with, generally, different number of positive and negative terms
\begin{equation}\label{17}
\dot{x}(t)+\sum_{j=1}^n a_j(t)x(h_j(t))-\sum_{i=1}^{\ell}  b_i(t)x(g_i(t))=0.
\end{equation}
Here $a_j(t)\geq 0$, $b_i(t)\geq 0$.

\begin{theorem}\cite{BB2009a}\label{th10}
Assume that  \eqref{2}  has a positive fundamental function and
$$ \displaystyle 
\sum_{j=1}^n a_j(t)-\sum_{i=1}^{\ell} b_i(t)\geq \alpha>0.
$$
%and the fundamental function of equation \eqref{2}  be positive. 
Then, equation \eqref{17} is UES.
\end{theorem}

 Further, we present several  results on  preservation of UES property for equations with multiple delays.
The first theorem is a corollary of a theorem obtained for an impulsive equation. Certainly, every statement
for equations with impulses also holds true for non-impulsive  equations.

Consider 
\begin{equation}\label{20}
\dot{x}(t)+\sum_{j=1}^n  a_j(t)x(h_j(t))+\sum_{i=1}^{\ell}  b_i(t)x(g_i(t))=0,
\end{equation}
where for the parameters of \eqref{20} the same conditions hold as for those  in  \eqref{2}.

\begin{theorem}\cite{BB1995}\label{th17}
Assume that equation \eqref{2} is UES. 
There is $\nu>0$  such that if
\begin{equation*}\label{21}
\limsup_{s \rightarrow\infty} \int_{s}^{s+1}\sum_{i=1}^{\ell} |b_i(\zeta)|d \zeta <\nu
\end{equation*}
then equation \eqref{20} is UES.
\end{theorem}

\begin{corollary}\label{cor3}
Assume that equation \eqref{2} is UES. Then any of the three conditions
%\\
%a) %There is a positive constant s $\nu>0$ for which
%\begin{equation}\label{22}
%The inequality
$$\displaystyle \lim_{s \rightarrow\infty} \int_{s}^{s+1}\sum_{i=1}^{\ell}  |b_i(\zeta)|d\zeta=0 , \quad 
\lim_{s \rightarrow\infty} \sum_{i=1}^{\ell}  |b_i(s)|=0, \quad 
\int_{t_0}^{\infty}\sum_{i=1}^{\ell}  |b_i(\zeta)|d\zeta<\infty
$$
%\end{equation}
implies that %equation 
\eqref{20} is UES.
%\\
%b) If $\displaystyle  \lim_{s \rightarrow\infty} \sum_{i=1}^{\ell}  |b_i(s)|=0$
%~
%or
%~
%  $\displaystyle  \int_{t_0}^{\infty}\sum_{i=1}^{\ell}  |b_i(\zeta)|d\zeta<\infty$ ~
%then equation \eqref{20} is UES.
\end{corollary}

Let us investigate, together with  \eqref{2},  a model with the same number of delay terms
\begin{equation}\label{20a}
\dot{x}(t)+\sum_{j=1}^n b_j(t)x(g_j(t))=0.
\end{equation}
This is  a ``perturbation" of  \eqref{2}.
Assume also that for equation \eqref{20a} similar conditions as for \eqref{2} hold.

\begin{theorem}\cite{BB2011b}\label{th18}
Let equation \eqref{2} be  UES, $A(t)\geq a_0>0$,  where $A$ is defined in \eqref{A_definition}, and for some $\mu \in (0,1)$,
$$
\sum_{j=1}^n \left[\frac{|a_j(t)-b_j(t)|}{A(t)}+
\frac{|b_j(t)|}{A(t)}\left|\int_{h_j(t)}^{g_j(t)}\sum_{j=1}^n |b_j(s)| ds\right| ~ \right] \leq \mu<1.
$$
Then equation  \eqref{20a} is UES.
\end{theorem}

Next, consider delay perturbations for the same coefficients.

\begin{theorem}\cite{BB2011b}\label{th19}
Let the coefficients be nonnegative $a_j(t)\geq 0$ and their sum be separated from zero  $A(t)\geq a_0>0$, where $A$ is defined in \eqref{A_definition},  and \eqref{2}  be nonoscillatory, i.e. its fundamental function be  positive.
The inequality
$$
\limsup_{s \rightarrow\infty} \sum_{j=1}^n \frac{a_j(s)}{A(s)}\left|\int_{h_j(s)}^{g_j(s)} A(\zeta)d \zeta\right|<1
$$
yields that
$$
\dot{x}(t)+\sum_{j=1}^n a_j(t)x(g_j(t))=0
$$
is UES.
\end{theorem}

\subsection{Equations with two delayed terms}

In the following result, there are no limitations on the signs of coefficients, only on the sign of their sum.

\begin{theorem}\cite{BB2007}\label{th5}
Assume that there is a positive constant $a_0>0$ for which $\alpha(t)+\beta(t)\geq a_0$ and
$$
\int_{h(t)}^t (\alpha(\zeta)+\beta(\zeta))d\zeta \leq \frac{1}{e}, \quad \limsup_{t\rightarrow\infty}\frac{|\beta(t)|}{\alpha(t)+\beta(t)}
\left|\int_{h(t)}^{g(t)}(\alpha(\zeta)+\beta(\zeta))d \zeta\right|<1.
$$
Then the equation
%\begin{equation}\label{14}
$$
\dot{x}(t)+\alpha(t)x(h(t))+\beta(t)x(g(t))=0
$$
%\end{equation}
is UES.
\end{theorem}

Further, let one term be positive dominating, the other term be negative. We start with the case when the positive term also has a larger delay.

\begin{theorem}\cite{BB2009a}\label{th7}
Let
$$
\alpha(t)\geq \beta(t)\geq 0, \quad t\geq g(t)\geq h(t), 
\quad \limsup_{t\rightarrow\infty}\int_{h(t)}^t \left[\alpha(\zeta)-\frac{1}{e}\beta(\zeta)\right]d\zeta <\frac{1}{e}  \, .
$$
Then the fundamental function  of
\begin{equation}\label{15}
\dot{x}(t)+\alpha(t)x(h(t))-\beta(t)x(g(t))=0
\end{equation}
is positive (leading to the nonoscillatory property of the equation). If  the sum is not only nonnegative, but also is separated from zero $\alpha(t)-\beta(t)\geq a_0 >0$,
this leads to UES of  \eqref{15}.
\end{theorem}

Next, consider \eqref{15}  without delay hierarchy, as long as the difference between the two delays is small enough.

\begin{theorem}\cite{BB2019}\label{th12}
Let $\alpha(t)\geq 0$,  $\beta(t)\geq 0$ be such that their difference is separated from zero  $\alpha(t)- \beta(t)\geq a_0 >0$ and either
%and at least one of the following conditions holds:
\\
1)
$ \displaystyle 
\limsup_{t\rightarrow\infty}\int_{h(t)}^t [\alpha(\zeta)-\beta(\zeta)]d \zeta \leq \frac{1}{e},~
\left\|\frac{\beta}{\alpha-\beta}\right\| \left\|\int_{h(\cdot)}^{g(\cdot)}[\alpha(\zeta)-\beta(\zeta)]d\zeta \right\|<\frac{1}{2}
$
\\
or
\\
2)
$ \displaystyle 
\limsup_{t\rightarrow\infty}\int\limits_{h(t)}^t  [\alpha(\zeta)-\beta(\zeta)]d\zeta> \frac{1}{e},
\left\|\int_{h(\cdot)}^{(\cdot)} \!\!\! \!\! [\alpha(\zeta)-\beta(\zeta)]d\zeta \right\|+
2\left\|\frac{\beta}{\alpha-\beta}\right\| \left\|\int_{h(\cdot)}^{g(\cdot)}  \!\!\!\! [\alpha(\zeta)-\beta(\zeta)]d\zeta \right\|<1+\frac{1}{e},
$
\\
equation \eqref{15} is UES.
\end{theorem}

\subsection{Equations with a non-delay term}

Let one of the terms include no delays
\begin{equation}\label{19a}
\dot{x}(t)+\alpha(t)x(t)+\sum_{j=1}^n b_j(t) x(h_j(t))=0 \mbox{ ~ for ~ } t\geq t_0.
\end{equation}
Denote $b(t)=\sum_{j=1}^n b_j(t)$.

The result below generalises Theorem~\ref{th21}  to  the case of measurable parameters and several delays.

\begin{theorem}\cite{BB2015}\label{th22}
Let all the coefficients in \eqref{19a}  be nonnegative, $ \alpha(t)\geq a_0>0$ and 
$$
 \quad h_0=\limsup_{s \rightarrow\infty} 
\int_{\min_j \{h_j(s)\}}^s \alpha(\zeta)d\zeta <\infty, \quad \beta=\limsup_{s \rightarrow\infty} \frac{b(t)}{\alpha(t)}<\infty.
$$
If
$$
\frac{1}{\beta}e^{-h_0}>\ln\frac{\beta^2+\beta}{\beta^2+1} 
$$
then equation \eqref{19a} is AS.   
\end{theorem}

In Theorem \ref{th22} for  equations with positive coefficients, domination of the non-delay term is not mandatory. 
Next, consider the case when coefficients can be oscillatory.

\begin{theorem}\cite{BB2021}\label{th15}
Let  $\alpha(t)$ be nonnegative. Then the inequality
$$
 \limsup_{s \geq t_0} \int_{t_0}^s  e^{ -\int_{\zeta}^s  \alpha(\xi)d\xi} \sum_{j=1}^n |b_j(\zeta)|d\zeta \leq \gamma<1
$$
implies boundedness on $[t_0,\infty)$ of  all the  solutions to  \eqref{19a}.
\end{theorem}

Consider an auxiliary equation
\begin{equation}\label{19}
\dot{y}(t)+ (\alpha(t)-\lambda)y(t)+\sum_{j=1}^n b_j(t)e^{\lambda (t-h_j(t))} y(h_j(t))=0.
\end{equation}

\begin{theorem}\cite{BB2025}\label{th16}
Let, for a fixed $\lambda>0$, all solutions of \eqref{19} be bounded on $[t_0,\infty)$.
Then equation \eqref{19a} is UES.
\end{theorem}

\begin{corollary}\label{cor6} \cite{BB2025,BTT}
a) If  $\alpha(t)\geq 0$ and $\displaystyle \sum_{j=1}^n |b_j(t)|\leq \alpha(t)$ then all solutions of equation \eqref{19a} are bounded on $[t_0,\infty)$.

b) If   $\alpha(t)\geq a_0 > 0$ then for any $\gamma\in(0,1)$, the  inequality $\displaystyle  \sum_{j=1}^n |b_j(t)|\leq \gamma \alpha(t)$
implies UES of  \eqref{19a}.
\end{corollary}

\begin{example}\label{ex5}
The model with a nondelay term
\begin{equation}\label{26b}
\dot{x}(t)=- x(t) + b(t) x(h(t))=0,~x(0)=1
\end{equation}
will be used as a test equation, aimed to outline criteria in this section.

Let us start with Theorem~\ref{th22} and the particular case of the parameters
$$
b(t) = 1.8 + 0.2 \cos t, \quad \beta=2, \quad h(t) = t-0.25-0.07\sin t, \quad h_0=0.32.
$$
Then, as can be easily seen, the non-delay term does not dominate but
$$
\frac{1}{\beta}e^{-h_0} \approx 0.363 >\ln\frac{\beta^2+\beta}{\beta^2+1} \approx 0.358,
$$
thus by Theorem~\ref{th22},  \eqref{26b} is AS.

For $b$ of an arbitrary sign in \eqref{26b}, 
 Corollary~\ref{cor6}  implies that \eqref{26b} is UES  when $|b(s)| \leq \gamma <1$.

Finally, let us show  that  in the case of coefficients of arbitrary signs, integral domination of the non-delay term
\begin{equation}
\label{26B}
\int_0^{\infty}  (\alpha(\zeta) +|\beta(\zeta)| )~d\zeta = -\infty
\end{equation}
is {\bf not sufficient}  for asymptotic stability of equation \eqref{19a}.

To this end, consider  \eqref{26b}  where on $[0,1]$ we have $h(t)\equiv 0$, $b_1(t) \equiv 0.5$, $x(0)=1$,
leading to  $\dot{x}(t)=- x(t) + 0.5$, and the solution is $0.5+0.5e^{-t}$.
We also have $x(1) = (1+e^{-1})/2 \approx 0.684$, the positive solution is decreasing on $[0,1]$. Next, take
$h(t)\equiv 0$, $b_1(t) \equiv 2$ on $[1,T]$, where $T$ will be defined later such that $x(T)=1$.
The initial value problem
$$
\dot{x}(t)=- x(t) + 2, \quad x(1) = \frac{1}{2} + \frac{1}{2e}
$$
has the solution $\displaystyle x(t) =  2 - \left( \frac{3}{2} - \frac{1}{2e} \right) e^{-(t-1)}$ on some segment starting at one. We have
$$
 2 - \left( \frac{3}{2} - \frac{1}{2e} \right) e^{-(T-1)} = 1 ~\Leftrightarrow ~ e^{T-1} =  \frac{3}{2} - \frac{1}{2e} 
 ~\Leftrightarrow ~ T-1 = \ln  \left( \frac{3}{2} - \frac{1}{2e} \right)  \approx 0.275 <0.5.
$$
Hence, \eqref{26b} with the initial value  $x(0)=1$ and $T$-periodic parameters outlined as follows
$$
h(t)=0, \quad  b(t) = \left\{    \begin{array}{ll}  0.5, & t \in [0,1), \\
2, & t\in [1,T)  \end{array}  \right. 
$$
is also $T$-periodic. 
Here the delays are bounded by $T<1.5$, and the integral
$$
\int_0^{T} (\alpha(\zeta) +|b(\zeta)| )~d\zeta = \int_0^1 (-1+0.5)~d\zeta + \int_1^{T} (-1+2) ~d\zeta = -0.5 +   T -1 < -0.2 <0,
$$
thus condition \eqref{26B} holds but  \eqref{26b} is not AS, as it possesses a periodic solution.
\end{example}

\section{Application of a priori estimation} 
\label{sec:apriori}

In this section, we obtain UES tests by the use of a priori estimation for solutions to  linear equation \eqref{2}  and  their
derivatives.

We assume, as previously, measurability and boundedness of the coefficients and the delays and  %conditions (a1)–(a3) hold for all equations throughout this and the next section, and
also that $t-h_j(t)\leq \tau_j$, $j=1,\dots,n$.

Denote 
$$
A(t)=\sum_{j=1}^n a_j(t), \quad  J=[t_0,t_1], \quad \Omega\subset N_n=\{1,\dots,n \}, \quad A_{\Omega}(t)=\sum_{j \in \Omega} a_j(t),
$$
and the norms $\|f\|$, $\|f\|_J$ on the halfline and a segment are defined as above; as usual,
%$$
%\|f\|={\rm ess}\sup_{t\geq t_0} |f(t)|, \|f\|_J= {\rm ess} \sup_{t\in [t_0,t_1]} |f(t)|,
%$$
by $I$  we denote an %is the $2 \times 2$ 
identity matrix.

We recall \cite{Berman} that a matrix $B=(b_{ij})_{i,j=1}^m$ is an  $M$-matrix (here we reduce consideration to non-singular matrices) if its off-diagonal entries are non-positive,  and either $B$ is invertible with $B^{-1} \geq 0$ (meaning that all the entries of $B^{-1}$ are nonnegative)
or (which is equivalent) its leading principal minors are positive.
In particular, $2\times 2$ matrix $B$ is an $M$-matrix and is not singular if $b_{12}\leq 0$, $b_{21}\leq 0$,  $b_{11} > 0$, $b_{22} > 0$ and in addition, $\det B>0$.

\begin{theorem}\label{th2}
Let there be a non-empty set  $\Omega\subset N_n$ and a constant $\alpha>0$ such that $A_{\Omega}(t)\geq \alpha >0$ 
and either
\begin{equation}\label{6}
\sum_{j=1}^n  \tau_j  \|a_j\|<1,~
\sum_{j=1}^n \|a_j\| \sum_{j \in \Omega} \tau_j  \left\|\frac{a_j}{A_{\Omega}}\right\|
+\sum_{j \in N_n/\Omega}  \left\|\frac{a_j}{A_{\Omega}}\right\|<1
\end{equation}
or
\begin{equation}\label{7}
\|A\| \sum_{j \in \Omega} \tau_j \left\|\frac{a_j}{A_{\Omega}}\right\|<
\left(1-\sum_{j=1}^n  \tau_j  \|a_j\|\right) \left(1-\sum_{j \in N_n/\Omega}  
\left\|\frac{a_j}{A_{\Omega}}\right\|\right).
\end{equation}
Then equation \eqref{2} is UES.
\end{theorem}

\begin{proof}
1) Assume that $x(t)$  is a  solution of  \eqref{5} and obtain  two estimates for its derivative  $\dot{x}(t)$.
From equation \eqref{5} we have
\begin{equation}\label{8a}
\|\dot{x}\|_J\leq \sum_{j=1}^n \|a_j\|\|x\|_J+\|f\|.
\end{equation}
By the addition and subtraction of the term $A(t)x(t)$, \eqref{5}  can be presented as
\begin{equation}\label{8}
\dot{x}(t)+A(t)x(t)=\sum_{j=1}^n a_j(t)\int_{h_j(t)}^t \dot{x}(\xi)d\xi+f(t).
\end{equation}
Equality \eqref{8} implies
\begin{equation}\label{9}
\|\dot{x}\|_J\leq \|A\|\|x\|_J+\sum_{j=1}^n \tau_j \|a_j\|\|\dot{x}\|_J+\|f\|.
\end{equation}

2) Next, let us estimate  $x(t)$, first on $J$, then on the halfline.
To this end, rewrite  \eqref{5}  as
\begin{equation}\label{10}
\dot{x}(t)+A_{\Omega}(t)x(t)=\sum_{j\in \Omega} a_j(t)\int_{h_j(t)}^t \dot{x}(\xi)d\xi
+\sum_{j \in N_n/\Omega}a_j(t)x(h_j(t))+f(t).
\end{equation}

From \eqref{10} and the assumption  $x(t_0)=0$, we obtain
$$
x(t)=\int_{t_0}^t e^{-\int_{\zeta}^t A_{\Omega}(s)ds}A_{\Omega}(\zeta)\left(\sum_{j \in\Omega}\frac{a_j(\zeta)}{A_{\Omega}(\zeta)}\int_{h_j(\zeta)}^{\zeta} \dot{x}(\xi)d\xi
+\sum_{j \in N_n/\Omega}\frac{a_j(\zeta)}{A_{\Omega}(\zeta)}x(h_j(\zeta))\right)d\zeta +\tilde{f}(t),
$$
where $\tilde{f}(t)=\int_{t_0}^t e^{-\int_s^t A_{\Omega}(\zeta)d\zeta}f(s)ds$.

If $A_{\Omega}(t)\geq 0$, then  $\left|\int_{t_0}^t e^{-\int_s^t A_{\Omega}(\tau)d\tau}
 A_{\Omega}(s) u(s)~ds\right|\leq \| u\|$.
Hence  $\|\tilde{f}\|<\infty$ and
\begin{equation}\label{11}
\|x\|_J\leq \sum_{j \in\Omega}\tau_j \left\|\frac{a_j}{A_{\Omega}}\right\|\|\dot{x}\|_J
+\sum_{j\in N_n/\Omega}\left\|\frac{a_j}{A_{\Omega}}\right\|\|x\|_J+\|\tilde{f}\|.
\end{equation}

3)  Let us introduce the notation 
$$
X=\{\|x\|_J,\|\dot{x}\|_J\}^T,~ F=\{\|f\|,\|\tilde{f}\|\}^T,~
A_1=\left(\begin{array}{cc}
\sum_{j\in N_n/\Omega}\left\|\frac{a_j}{A_{\Omega}}\right\|&\sum_{j \in\Omega}\tau_j \left\|\frac{a_j}{A_{\Omega}}\right\|\\
\sum_{j=1}^n \|a_j\|&0\end{array}\right),
$$$$
B_1=I-A_1=\left(\begin{array}{cc}
1-\sum_{j \in N_n/\Omega}\left\|\frac{a_j}{A_{\Omega}}\right\|&-\sum_{j \in\Omega}\tau_j \left\|\frac{a_j}{A_{\Omega}}\right\|\\
-\sum_{j=1}^n \|a_j\|&1\end{array}\right).
$$
Inequalities \eqref{11} and \eqref{8a} imply  $X\leq A_1X+ F$, therefore $B_1X\leq F$.
If condition \eqref{6} holds then the matrix $B_1$ is non-sinular and is also an $M$-matrix. Hence $X\leq B_1^{-1}F$, and
$B_1^{-1}$ is a matrix with non-negative entries, none of which depends on the interval $J=[t_0,t_1]$.
Then $\|x\|<\infty$, and by Lemma~\ref{lemma1} equation \eqref{2} is UES.

Assume now that condition \eqref{7} holds and denote
$$
A_2=\left(\begin{array}{cc}
\sum_{j \in N_n/\Omega}\left\|\frac{a_j}{A_{\Omega}}\right\|&\sum_{j \in\Omega}\tau_j \left\|\frac{a_j}{A_{\Omega}}\right\|\\
\|A\|&\sum_{j=1}^n \tau_j \|a_j\|\end{array}\right),~B_2=I-A_2.
$$
Inequalities \eqref{11} and \eqref{9}  imply   $X\leq A_2X+ F$ and thus $B_2 X\leq F$.
If condition \eqref{7} holds then the matrix $B_2$ is non-singular and is also an $M$-matrix. 
Hence $X\leq B_2^{-1}F$,  and the inverse  $B_2^{-1}$ has non-negative entries, none of which depends on the initial interval $J=[t_0,t_1]$.
Then $\|x\|<\infty$, and Lemma~\ref{lemma1} yields that equation \eqref{2} is UES.
\end{proof}

Choosing  $\Omega=N_n$ leads to the following statement. 

\begin{corollary}\label{cor2}
Let  $A(t)\geq \alpha>0$ and either
\vspace{2mm}
\\
%\begin{equation}\label{11a}
(i) $\displaystyle \sum_{j=1}^n \tau_j \|a_j\|<1,~ \sum_{j=1}^n \|a_j\|\sum_{j=1}^n \tau_j \left\|\frac{a_j}{a}\right\|<1$ ~ 
%\end{equation}
\vspace{2mm} \\
or ~ \vspace{2mm} \\
%\begin{equation}\label{11b}
(ii) $\displaystyle  \|A\|\sum_{j=1}^n \tau_j \left\|\frac{a_j}{a}\right\|<1-\sum_{j=1}^n \tau_j \|a_j\|$. \vspace{2mm}
%\end{equation}

Then, equation \eqref{2} is UES.
\end{corollary}

Condition \eqref{6} contains $2^n-1$ independent stability tests. If the set $\Omega$ consists of one element
then the condition in \eqref{6} is better than the corresponding condition in \eqref{7}. Hence condition b) has only
$2^n-n-1$ independent from a) stability conditions,  therefore conditions a) and b) together have $2^{n+1}-n-2$
independent stability tests. For example, for $n=2$ there are 4 independent stability tests outlined below.

\begin{corollary}\label{cor1}
Consider a model with two delay terms
\begin{equation}\label{12}
\dot{x}(t)+a_1(t)x(h_1(t))+a_2(t)x(h_2(t))=0,
\end{equation}
assuming for the parameters of equation \eqref{12} the same conditions as those  for equation \eqref{2}.
Let anyone of the four  conditions hold:
\\
1) $\displaystyle a_1(t)\geq \alpha>0,~ \tau_1(\|a_1\|+\|a_2\|)+\left\|\frac{a_2}{a_1}\right\|<1;$
\\
2)  $\displaystyle  a_2(t)\geq \alpha>0,~ \tau_2(\|a_1\|+\|a_2\|)+\left\|\frac{a_1}{a_2}\right\|<1;$
\\
3) $\displaystyle  a_1(t)+a_2(t)\geq \alpha>0,~
(\|a_1\|+\|a_2\|)\left(\tau_1\left\|\frac{a_1}{a_1+a_2}\right\|+\tau_2\left\|\frac{a_2}{a_1+a_2}\right\|\right)<1;$
\\
4) $\displaystyle  a_1(t)+a_2(t)\geq \alpha>0,~
\|a_1+a_2\|\left(\tau_1\left\|\frac{a_1}{a_1+a_2}\right\|+\tau_2\left\|\frac{a_2}{a_1+a_2}\right\|\right)
<1-\tau_1\|a_1\|-\tau_2  \|a_2\|  $.

Then equation \eqref{12} is UES.
\end{corollary}

Consider  equation \eqref{19a} with a non-delay term. 
The following corollary gives a delay-independent exponential stability condition.
%There are several such conditions where non-delay term dominates on delay terms, which slightly differ from conditions of Corollary \ref{cor1a}.

\begin{corollary}\label{cor1a}
Assume that $a(t)\geq a_0>0$ and either
\vspace{2mm}  \\
a)
$\displaystyle \sum_{j=1}^n \left\|\frac{b_j}{\alpha}\right\|<1$  \vspace{2mm}
\\
or   \vspace{2mm}
\\
b) for some $t_1>t_0$ and $0<\gamma<1$, we have $\displaystyle  \sum_{j=1}^n |b_j(t)|\leq \gamma \alpha (t)$.
\vspace{2mm}

Then, equation \eqref{19a} is UES.
\end{corollary}

\begin{example}\label{example1}
Let both of the two delays be constant, while one of the two coefficients be oscillatory
\begin{equation}\label{13}
\dot{x}(t)+(1-2\cos^2 t)x(t-0.2)+2\cos^2 t ~x(t-0.1)=0, \quad t\geq 0.
\end{equation}
Let us verify each of the hypotheses~1) - 4) in  Corollary~\ref{cor1} for equation \eqref{13}.
Since the inequalities $a_i(t)\geq \alpha>0$ are not satisfied for either $i=1$ or $i=2$, both conditions 1) and 2) do not hold.
We have
$$
a_1(t)+a_2(t)=1>0, \quad  \|a_1\|=\left\|\frac{a_1}{a_1+a_2}\right\|=1, \quad \|a_2\|=\left\|\frac{a_2}{a_1+a_2}\right\|=2,
$$
hence the second condition in 3) has the form $3(0.2+0.2)<1$ and also fails.

Condition 4)  has the form $0.2+0.2<1-0.2-0.2$  and therefore is satisfied.
Hence  by Part 4) of Corollary~\ref{cor1}, equation \eqref{13} is UES.
\end{example}

It is not difficult to give examples when one of the conditions 1) - 3) in Corollary~\ref{cor1}  holds, when the other  three conditions fail.
So conditions 1) - 4) are independent.

\section{Applications to nonlinear equations}
\label{applications}

To illustrate applications of stability tests obtained so far in the linear case,
take   a variation of the classical Mackey-Glass equation as a model.

\begin{example}\label{ex2}
We investigate 
\begin{equation}\label{25a}
\dot{x}(t)+\alpha(t)x(h_1(t))\left(1+\frac{1}{1+x^2(h_1(t))}\right)+\beta(t)\frac{x(h_2(t))}{1+x^2(h_2(t))}=0
\end{equation}
with measurable essentially bounded  coefficients $\alpha(t)$ and $\beta(t)$, and bounded delays $0\leq t-h_i(t)\leq \tau_i$,  $i=1,2$.

To get LES sufficient conditions, denote
$$
f_1(t,u)=\alpha(t)u\left(1+\frac{1}{1+u^2}\right), \quad f_2(t,u)=\beta(t)\frac{u}{1+u^2}
$$
and 
$$
a(t)=\lim_{u\rightarrow 0} \frac{f_1(t,u)}{u}=\alpha(t), \quad  b(t)=\lim_{u\rightarrow 0} \frac{f_2(t,u)}{u}=\beta(t).
$$
Then the linearised equation to \eqref{25a} is 
\begin{equation}\label{26a}
\dot{y}(t)+\alpha(t)y(h_1(t))+\beta(t)y(h_2(t))=0.
\end{equation}

From Corollaries~\ref{cor3a} and \ref{cor3b}, respectively,  the following results are obtained.

\begin{statement}\label{st1}
Let  the sum of the coefficients $\alpha(t)+\beta(t)\equiv a_0>0$ be constant, while
$$ \displaystyle  \tau_1\|\alpha\|+\tau_2\|\beta\|<1+\frac{1}{e} \,. $$
Then, equation \eqref{25a} is LAS.
\end{statement}

\begin{statement}\label{st2}
Let $ 0<a_0\leq \alpha(t)+\beta(t)\leq A_0$ and
%at least one of the following assumptions hold:
either  \vspace{2mm}
\\
a) $\displaystyle \tau_1\|\alpha\|+\tau_2\|\beta\|<\frac{a_0}{A_0} \left( 1+\frac{1}{e}  \right)$  \vspace{2mm}
\\
or  \vspace{2mm}
\\
b)  $\displaystyle 
\limsup_{s \rightarrow\infty}\left(\int_{h_1(s)}^s \left[ \alpha(\zeta)+\beta(\zeta) \right]\, d\zeta+\int_{h_2(s)}^s  \left[ \alpha(\zeta)+\beta(\zeta) \right] \, d\zeta\right)
<1+\frac{1}{e}\,.
$
\vspace{2mm}

Then, equation \eqref{25a} is LAS.
\end{statement}

Next, we will apply  Corollary~\ref{cor1}.

\begin{statement}\label{st3}
Assume that anyone of the four hypotheses is valid:

a)  $\displaystyle 
\alpha(t)\geq \alpha_0>0, ~ \tau_1(\|\alpha\|+\|\beta\|)+\left\|\frac{\beta}{\alpha}\right\|<1
$;

b)   $\displaystyle 
\beta(t)\geq \beta_0>0, ~ \tau_2(\|\alpha\|+\|\beta\|)+\left\|\frac{\alpha}{\beta}\right\|<1;
$

c)  $\displaystyle \alpha(t)+\beta(t)\geq a_0>0, (\|\alpha\|+\|\beta\|)\left(\tau_1\left\|\frac{\alpha}{\alpha+\beta}\right\|
+\tau_2\left\|\frac{\beta}{\alpha+\beta}\right\|\right)<1$;

d)  $\displaystyle  \alpha(t)+\beta(t)\geq a_0>0, \|\alpha+\beta\|\left(\tau_1\left\|\frac{\alpha}{\alpha+\beta}\right\|
+\tau_2\left\|\frac{\beta}{\alpha+\beta}\right\|\right)<1-\tau_1\|\alpha\|-\tau_2\|\beta\| $.

Then equation \eqref{25a} is LES.
\end{statement}
 
 \end{example}

Sometimes for equation \eqref{1}, we  are interested only in solutions in a certain interval  $A\leq x(t)\leq B$. In this
case, for GES we will consider only such solutions. This in particular is true for population dynamics equations, where only positive solutions are biologically senseful, 
and there is also a some natural upper bound.
% are non-negative originally.

If  linearised equation \eqref{2} is  UES,  we can conclude that  the zero equilibrium (or any other constant equilibrium, we can make a shift to zero) of equation \eqref{1} is LES.
%is locally exponentially stable.  
The situation with global exponential stability (GES) is more complicated. 
One of the methods named global linearised stability \cite{BB2009b}, can be applied by  construction  another linear equation, 
associated with nonlinear equation \eqref{1}.
Let $x(t)$ be an arbitrary  fixed solution of equation \eqref{1}. 
Denote
$$
c_j(t)=\left\{\begin{array}{cc}
\frac{f_j(t,x(h_j(t)))}{x(h_j(t))},& x(h_j(t))\neq 0,\\
0, &x(h_j(t))=0,\end{array}\right. 
$$
and consider the associated linear equation
\begin{equation}\label{3}
\dot{y}(t)+\sum_{j=1}^n c_j(t)y(h_j(t))=0.
\end{equation}

If equation \eqref{3} is UES, then $x(t)$ has an exponential estimate, as it solves \eqref{3}. Hence every solution
of \eqref{1} has an exponential estimate. It means that equation \eqref{1} is GES.

More precisely, the global linearised stability method \cite{BB2009b} is formulated below.

\begin{theorem}\cite{BB2009b}\label{th1}
Assume that  there exist  numbers
$$
-\infty \leq A  \leq   B   \leq +\infty, ~-\infty \leq d_j\leq D_j \leq +\infty, ~j=1,\dots, n,
$$
such that any solution $x(t)$ of equation \eqref{1} has a priori estimates 
$$
A   \leq \liminf_{s \rightarrow\infty}x(s)\leq \limsup_{s\rightarrow\infty}x(s)\leq   B
$$
and 
\begin{equation}\label{4}
d_j \leq \frac{f_j(t,u)}{u}\leq D_j,~ u\in [A,B], \quad j =1, \dots, n.
\end{equation}

If equation \eqref{3} is UES  for any $c_j(t)$ satisfying the inequalities 
$
d_j\leq c_j(t)\leq D_j,
$
then \eqref{1} is GES.
\end{theorem}

\begin{example}\label{example3}
Let us return to model  \eqref{25a}.
To obtain global stability result for  \eqref{25a}, we will apply Theorem \ref{th1}.
Take $x(t)$ which  solves \eqref{25a}. and
denote 
\begin{equation}\label{27abc}
c(t):=\frac{f_1(t,x(h_1(t)))}{x(h_1(t))}=\alpha(t)\left(1+\frac{1}{1+x^2(h_1(t))}\right),
~~  d(t):=\frac{f_2(t,x(h_2(t)))}{x(h_2(t))}=\frac{\beta(t)}{1+x^2(h_2(t))}.
\end{equation}
Next,  consider a linear counterpart of \eqref{25a}
\begin{equation}\label{27a}
\dot{z}(t)+c(t)z(h_1(t))+d(t)z(h_2(t))=0
\end{equation}
with the above notation.

Since the inequality $d(t)\geq d_0>0$ fails, both Corollary~\ref{cor3a} and Part b) of Corollary~\ref{cor1}
fail  for equation \eqref{27a}. In the statement below we will apply Part~1 of Corollary~\ref{cor1}
(Parts~3  and 4 can be considered in a similar way) and also Corollary~\ref{cor3b}.

\begin{statement}\label{st4}
Assume that %at least one of the following conditions holds for some 
there exists a positive lower bound  $\alpha_0>0$ such that either  \vspace{2mm}
\\
a) $\displaystyle 
\alpha(t)\geq \alpha_0,  \quad \tau_1(2\|\alpha\|+\|\beta\|)+\left\|\frac{\beta}{\alpha}\right\|<1$  \vspace{2mm}
\\
or  \vspace{2mm}
\\
b) $\displaystyle 
\alpha(t)+\beta(t)\geq a_0,  \quad  (\tau_1+\tau_2)(2\|\alpha\|+\|\beta\|)<1+\frac{1}{e}$\,.
\vspace{2mm}

Then,  \eqref{25a} is GES.
\end{statement}

\begin{proof}
a)
We have for \eqref{27a},  taking into account  \eqref{27abc},
$$
c(t)\geq \alpha(t)\geq \alpha_0>0, ~~ \|c\|=2\|\alpha\|, \|d\|=\|\beta\|,~~ \left\|\frac{d}{c}\right\|
=\left\|\frac{\beta}{\alpha}\right\|.
$$
Part 1) of Corollary \ref{cor1} implies that equation \eqref{27a} is UES. By Theorem \ref{th2},
equation \eqref{25a} is globally exponentially stable.

b) We have 
$$
0<\alpha_0\leq \alpha(t)\leq c(t)+d(t)\leq 2\|\alpha\|+\|\beta\|.
$$
Hence
$$
\limsup_{t\rightarrow\infty}\left(\int_{h_1(t)}^t (c(s)+d(s))ds+\int_{h_2(t)}^t (c(s)+d(s))ds\right)
\leq (2\|\alpha\|+\|\beta\|)(\tau_1+\tau_2)<1+\frac{1}{e} \, .
$$
Inequality \eqref{17b} holds for equation \eqref{27a}. By Corollary \ref{cor3b}, this equation is UES.
Hence equation \eqref{25a} is GES.
\end{proof}

\end{example}

In Example \ref{example3}, GES conditions were obtained without applications of a priori 
estimation of solutions. In this case we estimated the coefficients of  linear equation \eqref{27a} that depend on the unknown solution over an infinite interval. If an approximate estimate of  solutions is known,  the infinite interval can be replaced by a finite one. 
This leads to more accurate and efficient stability tests for nonlinear equations.

Examples of the use of a priori estimates, including the three classical models of Mackey-Glass equations, can be found in 
the papers \cite{BB2009b, BB2012, BB2013a, BB2013b}.

\section{Discussion and open problems}
\label{sec:discussion}

Mathematical models described by differential equations with one or multiple delays  have applications in many
areas of natural sciences and engineering. One of important problems for  these models is to obtain explicit simply verifiable
local and global stability conditions. In the last two decades  many interesting results
for all kinds of functional differential equations appeared. In addition to the works cited in this article, we can also mention the following interesting papers on the stability of linear and nonlinear delay differential equations
\cite{alex1,alex2,barreira1,barreira2,chermak,Dix, dom1,dom2, GD, gil,gus},\cite{mal1}-\cite{mal6},
\cite{ngoc1,hu,LizPituk, stav, Tunc1, Tunc2,yskak}, and some earlier results \cite{GP,GD, kolmMysh}.

Explicit stability conditions for linear for systems of delay differential equations
one can find in \cite{dom4,dom5,faria2,tunc3,tunc5,greb,hu,maly,egorov,mul1,mul2,tian,wang,zhang1}
and on stability  of integro-differential equations and  equations with distributed delays \cite{dom2,dom6,dom8},
\cite{ngoc2}-\cite{ngoc7},\cite{tunc1,tunc4}.

Unfortunately, there are no monographs that correspond to the current state-of-arts in this area. 
There are no review articles on stability for individual types of equations.
In this paper we partially fill the gap for scalar differential equations with several delays.
Together with a review, we also obtain several new exponential stability conditions.
These new UES tests are quite simple and convenient in applications.

To apply global linearised stability method for a given nonlinear differential equation,
we have to estimate coefficients of a corresponding linear delay differential equation dependent, generally,
on an unknown solution of the given nonlinear equation. In  Section~\ref{applications}, we get such estimates for
the coefficients of \eqref{25a} without any information about solutions. 
However, for many known mathematical models, in order 
to estimate the coefficients of linear equation, we need first to construct  a priori  estimates for solutions
of a given nonlinear equation:
$$
A  \leq \liminf_{t\rightarrow\infty} x(t)\leq \limsup_{t\rightarrow\infty} x(t)\leq B.
$$
Examples of applications of global linearised stability method for some known mathematical models
can be  found in \cite{BB2009b,BB2013a,BB2013b,BB2025b}.

In this paper we separate stability tests for equations with measurable and  continuous parameters.
Actually, any stability test for equations with measurable parameters is also valid,
once we reduce ourselves to  continuous delays and coefficients, but the converse, generally, is not true. 
We investigate here equations, where  both coefficients and delays are bounded, which is appropriate to study GES and LES, but is not the general case.
New results on  stability for equations with unbounded coefficients and/or delays with
a review on known results can be found in the recent paper \cite{BB2025}.

In line with the review, we list some open problems arising from it. 

\begin{itemize}
\item
Is it possible to extend the results of  Theorem~\ref{th20} (with the same constant) to the case when coefficients and delays are allowed to be discontinuous (are measurable)? 
Note that Theorem~\ref{th20a} was extended to measurable parameters (and more general equations, including those with a distributed delay) not long ago, see
%For  $m=1$ and a variable coefficient the positive answer was justified in an integral form in papers 
\cite{Malygina, stav} and references therein.
\item
The next open problem is to extend (if possible) pointwise condition \eqref{23a} to more general integral form 
$$
\limsup_{t \to \infty} \sum_{j=1}^n \int_{h_j(t)}^t a_j(\zeta) d\zeta \leq  M <1.
$$
Does this inequality imply uniform exponential stability of equation \eqref{2}? And if not, what is the best constant?
\item
All the applications of global linearised stability method so far have been obtained for scalar equations.
How can this method be extended to vector equations and systems of scalar equations, as well as
equations of higher orders, and models with impulses?
\item
Consider equation \eqref{2d} where coefficients can have arbitrary signs, or be oscillatory.
Assume that $\displaystyle \liminf_{t\rightarrow\infty}\int_{h(t)}^t a(\zeta)d\zeta>0$. 
Is there a constant $M>0$ such that the condition $$\limsup_{t\rightarrow\infty}\int_{h(t)}^t a(\zeta)d\zeta <M$$
implies AS of equation \eqref{2d}?
\item
Can Theorem~\ref{th23c} be generalized (with $\limsup$ or/and $\liminf$ replacing the exact values) to the model, where both coefficients are variable, not only the delayed one, and the nondelay term also has a variable coefficient
$$
\dot{x}(t)=- \alpha(t)x(t)+\beta(t)x(h(t))?
$$
 %with variable coefficients and delay?
\end{itemize}

\section*{Acknowledgment}

The second author (E. Braverman) was partially supported by the NSERC Discovery Grant \#  RGPIN-2020-03934.


\begin{thebibliography}{99}

\bibitem{alex1} 
Alexandrova, Irina; Mondié, Sabine,
 On the stability of linear time-delay systems with arbitrary delays,
Lect. Notes Control Inf. Sci. Proc.
Springer, Cham, 2022, 73--80.

\bibitem{alex2}
I. Alexandrova, S. Mondié,
 Necessary stability conditions for linear systems with incommensurate delays,
Automatica J. IFAC 129 (2021), Paper No. 109628, 8 pp.

\bibitem{AS}
N.V. Azbelev, P.M. Simonov, Stability of Differential
Equations with Aftereffect. {\em Stability and Control:
Theory, Methods and Applications}, {\bf 20}. Taylor $\&$ Francis, London, 
2003. %222 pp.

\bibitem{BBbook}
R.P. Agarwal, L. Berezansky, E. Braverman, A. Domoshnitsky, 
Nonoscillation theory of functional differential equations with applications,
Springer, New York, 2012. %, xvi+520 pp. 

\bibitem{barreira1},
L. Barreira, C. Valls,
 Delay-difference equations and stability,
J. Dynam. Differential Equations 37 (2025), no. 1, 95--113.

\bibitem{barreira2}
L. Barreira, C.  Valls, 
Stability of delay equations,
Electron. J. Qual. Theory Differ. Equ. 2022, Paper No. 45, 24 pp.

\bibitem{BB1995}
L. Berezansky, E. Braverman,
Preservation of the exponential stability under perturbations of linear delay impulsive differential equations
Z. Anal. Anwendungen 14 (1995), %no. 1, 
157--174.

\bibitem{BB2006a}
L. Berezansky, E. Braverman,
On stability of some linear and nonlinear delay differential equations,
{%\em 
J. Math. Anal. Appl.} {%\bf 
314} %(2) 
(2006), 391--411.

\bibitem{BB2006b}
L. Berezansky, E. Braverman,
On exponential stability of linear differential equations with several
delays, {%\em 
J. Math. Anal. Appl.} {%\bf 
324} %(2) 
(2006), 1336-1355.

\bibitem{BB2007}
L. Berezansky, E. Braverman,
Explicit exponential stability conditions for linear 
differential equations with several delays, 
{%\em 
J. Math. Anal. Appl.} {%\bf 
332}  %(1) 
(2007), 246--264.  

\bibitem{BB2009a}
L. Berezansky, E. Braverman, 
Nonoscillation and exponential
stability of delay differential equations with oscillating coefficients,
{%\em 
J. Dyn. Control Syst.} {%\bf 
15 %(1)
} (2009),  63--82.

\bibitem{BB2009b}
L. Berezansky, E. Braverman,
 Global linearized stability theory for delay differential equations,
Nonlinear Anal. 71 (2009), %no. 7--8, 
2614--2624.

\bibitem{BB2011}
L. Berezansky, E. Braverman, 
New stability conditions for linear differential equations with several
delays, Abstr. Appl. Anal. 2011, Art. ID 178568, 19 pp.

\bibitem{BB2011b}
L. Berezansky, E. Braverman,  
 Preservation of exponential stability for equations with several delays
Math. Bohem. 136 (2011), %no. 2, 
135--144.

\bibitem{BB2012}
L. Berezansky, E. Braverman, L. Idels, 
The Mackey-Glass model of respiratory dynamics: review and new results,
Nonlinear Anal. 75 (2012), %no. 16, 
6034–--6052.

\bibitem{BB2013a}
L. Berezansky, E. Braverman, L. Idels, 
 Mackey-Glass model of hematopoiesis with non-monotone feedback: stability, oscillation and control,
Appl. Math. Comput. 219 (2013), %no. 11, 
626--6283.

\bibitem{BB2013b}
L. Berezansky, E. Braverman, L. Idels,
 Mackey-Glass model of hematopoiesis with monotone feedback revisited
Appl. Math. Comput. 219 (2013), %no. 9, 
4892--4907.

\bibitem{BB2015}
L. Berezansky, E. Braverman, 
 Stability conditions for scalar delay differential equations with a non-delay term,
Appl. Math. Comput. 250 (2015), 157--164.

\bibitem{BB2019}
L. Berezansky, E. Braverman, 
On stability of delay equations with
positive and negative coefficients
with applications,
Z. Anal. Anwendungen 
%Journal of Analysis and its Applications
38 (2019), 157--189.

\bibitem{BB2021}
L. Berezansky, E. Braverman, 
On exponential stability of linear delay differential equations
with oscillatory coefficients and kernels,
Differential  Integral Equations, 35 %(9-10) 
(2022), 559--580.

\bibitem{BB2025b}
L. Berezansky, E. Braverman, 
 Nicholson's blowflies differential equations with a small delay in the mortality term,
Nonlinear Anal. Real World Appl. 81 (2025), paper \# 104193, 10 pp.


\bibitem{BB2025}
L. Berezansky, J. Diblík, Z. Svoboda, Z. Šmarda, 
Tests for boundedness and exponential stability of linear integro-differential equations with unbounded delays,
Differential Integral Equations 38 (2025), %no. 1-2, 
43--70.


\bibitem{dom2}
L. Berezansky, A.  Domoshnitsky, O. Kupervasser, 
Bounded solutions and exponential stability for linear integro-differential equations of Volterra type,
Appl. Math. Lett. 154 (2024), Paper No. 109112, 6 pp.

\bibitem{Berman}
A. Berman, R. Plemmons,  Nonnegative Matrices in the Mathematical Sciences, 
Computer Science and Applied Mathematics, Academic
Press, New York-London, 1979.

\bibitem{BTT}
E. Braverman, C. Tun\c{c}, O. Tun\c{c},
On global stability of nonlinear systems with unbounded and
distributed delays and a dominating non-delay term,
Commun. Nonlinear Sci. Numer. Simul.   143 (2025), Paper $\#$ 108590, 17 pp.

\bibitem{chermak} 
J. Čermák, L. Nechvátal,
 On stability of linear differential equations with commensurate delayed arguments,
Appl. Math. Lett. 125 (2022), Paper No. 107750, 8 pp.


\bibitem{Dix}
 J.G. Dix,
 Asymptotic behavior of solutions to a first-order differential equation with variable delays,
Comput. Math. Appl. 50 (2005), %no. 10-12, 
1791--1800.

\bibitem{dom1} 
A. Domoshnitsky, E. Berenson, S. Levi, E. Litsyn, 
 Floquet theory and stability for a class of first order differential equations with delays,
Georgian Math. J. 31 (2024), no. 5, 757--772.



\bibitem{dom4}
A. Domoshnitsky, M. Gitman, R.  Shklyar,
 Stability and estimate of solution to uncertain neutral delay systems,
Bound. Value Probl. 2014, 2014:55, 14 pp.

\bibitem{dom5}
A. Domoshnitsky, R. Shklyar, M. Gitman, V. Stolbov,
Positivity of fundamental matrix and exponential stability of delay differential system,
Abstr. Appl. Anal. 2014, Art. ID 490816, 9 pp.

\bibitem{dom6}
A. Domoshnitsky, 
On stability of nonautonomous integro-differential equations,
World Scientific Publishing Co. Pte. Ltd., Hackensack, NJ, 2005, 1059--1064.

\bibitem{dom7}
A. Domoshnitsky, Ya.M.  Goltser,
Approach to study of bifurcations and stability of integro-differential equations,
Math. Comput. Modelling 36 (2002), no. 6, 663--678.


\bibitem{dom8}
A. Domoshnitsky, Ya.M.  Goltser,
One approach to study of stability of integro-differential equations,
Nonlinear Anal. 47 (2001), no. 6, 3885--3896.



\bibitem{Liz}
 A. Ivanov, E. Liz, S. Trofimchuk,
 Halanay inequality, Yorke 3/2 stability criterion, and differential equations with maxima, Tohoku Math. J. 54 (2002)
 277--295.
 
 \bibitem{Faria2022}
T. Faria,
 Stability for nonautonomous linear differential systems with infinite delay,
J. Dyn. Diff.  Equations 34 (2022), no. 1, 747--773.

\bibitem{faria2} 
T. Faria,
 Stability for nonautonomous linear differential systems with infinite delay,
J. Dynam. Differential Equations 34 (2022), no. 1, 747773.


\bibitem{Gil}
M.I. Gil', 
Stability of delay differential equations with oscillating coefficients, 
Electron. J. Differential Equations 2010, No. 99, 5 pp.

\bibitem{Gil2}
M.I. Gil', 
Stability of functional differential equations with oscillating coefficients and distributed delays, Differ. Equ. Appl. 3 (1) (2011) 11--19.

\bibitem{Gil3}
M. Gil',
Stability of Vector Differential Delay Equations.
Front. Math., Birkhäuser/Springer Basel AG, Basel, 2013.



\bibitem{gil}
M. Gil',
Delay-dependent stability conditions for non-autonomous 
functional differential equations with several delays in a Banach space,
Nonauton. Dyn. Syst. 8 (2021), no. 1, 168--179.

\bibitem{tunc3}
M. Gözen, C. Tunç,
A new result on exponential stability of a linear differential system of first order with variable delays,
Nonlinear Stud. 27 (2020), no. 1, 275--284.

\bibitem{tunc5}
M. Gözen, C. Tunç,
A note on the exponential stability of linear systems with variable retardations,
Appl. Math. Inf. Sci. 11 (2017), no. 3, 899--906.


\bibitem{greb}
B.G. Grebenshchikov, A.B.  Lozhnikov,
Asymptotic properties of a class of systems with linear delay,
Differ. Equ. 59 (2023), no. 5, 577--590.

%\bibitem{DomGus}
%S. Gusarenko, A. Domoshnitskiĭ, 
%Asymptotic and oscillation properties of first-order linear scalar functional-differential equations,
%Differential Equations 25 (1989), %no. 12, 
%1480--1491.

\bibitem{gus}
S.A. Gusarenko,
 Stability criterion for linear differential equations with a delayed argument,
Izv. Vyssh. Uchebn. Zaved. Mat. 2022, no. 12, 34–56; translation in
Russian Math. (Iz. VUZ) 66 (2022), no. 12, 33--52

\bibitem{GD}
S.A. Gusarenko, A.I. Domoshnitsky,  Asymptotic and  
oscillation properties of first-order linear scalar
functional-differential equations,
{%\em 
Differential Equations} {%\bf 
25} (12) (1989) 1480--1491.


\bibitem{GH2001}
I. Győri, F. Hartung, 
 Stability in delay perturbed differential and difference equations,
Fields Inst. Commun., 29,
American Mathematical Society, Providence, RI, 2001, 181--194.


\bibitem{GH}
I. Gy\"{o}ri, L. Horváth,
 Sharp estimation for the solutions of delay differential and Halanay type inequalities,
Discrete Contin. Dyn. Syst. Ser. A 37, (2017) 3211--3242. 

\bibitem{GP}
I. Gy\"{o}ri, M. Pituk, Stability criteria for linear delay 
differential equations, {%\em 
Differential Integral Equations} { %\bf 
10} (1997) 841--852.



\bibitem{Hale}
J.\,K. Hale, S.\,M.  Verduyn Lunel, Introduction to
Functional Differential equations. {\em Applied Mathematical Sciences}, 
{\bf 99}. Springer-Verlag, New York, 1993.

\bibitem{hu}
G.D. Hu, %Guang-Da,
Delay-dependent stability conditions through the fundamental matrix of solutions for linear delay differential systems,
J. Comput. Appl. Math. 475 (2026), Paper No. 117041, 8 pp.

\bibitem{kolmMysh}
{V.B. Kolmanovskii, A.D. Myshkis}.
Introduction to the Theory and Applications of Functional-Differential Equations. 
Mathematics and its Applications {\bf 463}, Kluwer Academic Publishers, Dordrecht, 1999.


\bibitem{Kz}
T. Krisztin, On stability properties for one-dimensional
functional-differential equations,
{%\em 
Funkcial. Ekvac.} {%\bf 
34} (1991) 241--256.




\bibitem{LizPituk}
E. Liz, M. Pituk, 
Exponential stability in a scalar functional differential equation,
J. Inequal. Appl. 2006, Art. ID 37195, 10 pp.



\bibitem{Malygina}
V.V.  Malygina, 
The Myshkis  3/2  theorem and its generalizations,
Sib. Math. J. 64 (2023), %no. 6, 
1372--1384.



\bibitem{mal1}
V. Malygina, T.  Sabatulina,
On estimates of solutions to autonomous differential equations with aftereffect and coefficients of different signs,
Funct. Differ. Equ. 31 (2024), no. 3-4, 191--203.

\bibitem{mal2}
V.V. Malygina, K.M. Chudinov, 
About exact two-sided estimates for stable solutions to autonomous functional differential equations,
Sib. Math. J. 63 (2022), no. 2, 299--315.


\bibitem{mal3}
V.V. Malygina, K.M. Chudinov, 
Stability of solutions of differential equations with several variable delays. I
Russian Math. (Iz. VUZ) 57 (2013), no. 6, 21--31.

\bibitem{mal4}
V.V. Malygina, K.M. Chudinov, 
 Stability of solutions to differential equations with several variable delays. II
Russian Math. (Iz. VUZ) 57 (2013), no. 7, 1--12.

\bibitem{mal5}
V.V. Malygina, K.M. Chudinov, 
Stability of solutions to differential equations with several variable delays. III
Russian Math. (Iz. VUZ) 57 (2013), no. 8, 37--48.

\bibitem{mal6}
V.V. Malygina,
On the exact boundaries of the stability domain of linear differential equations with distributed delay,
Izv. Vyssh. Uchebn. Zaved. Mat. 2008, no. 7, 19--28; translation in
Russian Math. (Iz. VUZ) 52 (2008), no. 7, 15--23.


\bibitem{maly}
A. Malyshev, M. Sadkane,
On the stability radius for linear time-delay systems,
BIT 64 (2024), no. 1, Paper No. 5, 24 pp.

\bibitem{egorov}
S. Mondié, A. Egorov, M.A. Gomez,
Lyapunov stability tests for linear time-delay systems,
Annu. Rev. Control 54 (2022), 68--80.


\bibitem{mul1}
M. Mulyukov,
Stability of delay differential equation with a discrete retarded argument,
Funct. Differ. Equ. 31 (2024), no. 3-4, 205--221.

\bibitem{mul2}
M. Mulyukov,  
On the stability of one equation with a discrete retarded argument and a constant concentrated delay,
Izv. Vyssh. Uchebn. Zaved. Mat. 2023, no. 10, 90--94; translation in
Russian Math. (Iz. VUZ) 67 (2023), no. 10, 81--84.

\bibitem{ngoc1}
P.H.A. Ngoc,
 Exponential stability of coupled linear delay time-varying differential-difference equations,
IEEE Trans. Automat. Control 63 (2018), no. 3, 843--848.

\bibitem{ngoc2}
P.H.A. Ngoc, L.T. Hieu,
 On uniform asymptotic stability of nonlinear Volterra integro-differential equations,
Internat. J. Control 95 (2022), no. 3, 729--735.

\bibitem{ngoc3}
P.H.A. Ngoc, L.T. Hieu,
 Exponential stability of integro-differential equations and applications,
Appl. Math. Lett. 117 (2021), Paper No. 107127, 8 pp.

\bibitem{ngoc4} 
P.H.A.  Ngoc, T.B. Tran, C.T. Tinh, N.D. Huy,
 Scalar criteria for exponential stability of functional differential equations,
Systems Control Lett. 137 (2020), 104642, 7 pp.

\bibitem{ngoc5}
P.H.A. Ngoc, C.T. Tinh, T.B. Tran,
 Further results on exponential stability of functional differential equations,
Internat. J. Systems Sci. 50 (2019), no. 7, 1368--1377.

\bibitem{ngoc6} 
P.H.A. Ngoc, T. Thai Bao, T. Cao Thanh, H.  Nguyen Dinh,
 Novel criteria for exponential stability of linear non-autonomous functional differential equations,
J. Syst. Sci. Complex. 32 (2019), no. 2, 479--495.

\bibitem{ngoc7}
T.A. Tran, P.H.A.  Ngoc,
New stability criteria for linear Volterra time-varying integro-differential equations,
Taiwanese J. Math. 21 (2017), no. 4, 841--863.



\bibitem{qu}
M. Qu, H.  Matsunaga, 
Exact stability criteria for linear differential equations with discrete and distributed delays,
J. Math. Anal. Appl. 539 (2024), no. 2, Paper No. 128663, 15 pp.

\bibitem{shab}
T.L. Sabatulina, V.V.  Malygina, 
On stability of a differential equation with aftereffect,
Russian Math. (Iz. VUZ) 58 (2014), no. 4, 20--34.

\bibitem{SYC}
J.W.H. So, J.S. Yu, M.P. Chen, Asymptotic stability for scalar
delay differential equations, {\em Funkcial. Ekvac.} {%\bf 
39} (1996), 
1--17.

\bibitem{stav}
J.I. Stavroulakis, E. Braverman,
Oscillation, convergence, and stability of linear delay differential equations,
J. Differential Equations 293 (2021), 282--312.

\bibitem{Tunc1}
O. Tunç, 
Stability tests and solution estimates for non-linear differential equations.(English summary)
Int. J. Optim. Control. Theor. Appl. IJOCTA 13 (2023), %no. 1, 
92--103.

\bibitem{Tunc2}
C. Tunç, O. Tunç,
 New results on the qualitative analysis of integro-differential equations with constant time-delay,
J. Nonlinear Convex Anal. 23 (2022), %no. 3, 
435--448.


\bibitem{Tan}
T.  Taniguchi,
 Asymptotic behavior theorems for non-autonomous functional-differential equations via Lyapunov-Razumikhin method,
J. Math. Anal. Appl. 189 (1995), %no. 3, 
715--730.

\bibitem{TangZou}
X. Tang, X. Zou
Stability of scalar delay differential equations 
with dominant delayed terms,
Proc. R. Soc. Edinb.  A 133 (2003), 951--968.

\bibitem{tian}
J. Tian, Z. Ren, S. Zhong,
A new integral inequality and application to stability of time-delay systems,
%Tian, Junkang; Ren, Zerong; Zhong, Shouming
Appl. Math. Lett. 101 (2020), 106058, 7 pp.

\bibitem{tunc}
C. Tunç, S.A. Mohammed, 
On asymptotic stability, uniform stability and boundedness of solutions of
 nonlinear Volterra integro-differential equations,
Ukraïn. Mat. Zh. 72 (2020), no. 12, 17081720; 
translation in Ukrainian Math. J. 72 (2021), no. 12, 1976--1989.


\bibitem{tunc1}
C. Tunç,  İ. Akbulut,
Stability of a linear integro-differential equation of first order with variable delays,
Bull. Math. Anal. Appl. 10 (2018), no. 2, 19--30.


\bibitem{tunc4} 
C. Tunç,  İ. Akbulut,
%Tunç, Cemil; Akbulut, İrem,
 Stability of a linear integro-differential equation of first order with variable delays,
Bull. Math. Anal. Appl. 10 (2018), no. 2, 19--30.




\bibitem{wang} 
Y. Wang, H. Liu, X. Li,
%Wang, Yongqing; Liu, Haibo; Li, Xu.
 A novel method for stability analysis of time-varying delay systems,
IEEE Trans. Automat. Control 66 (2021), no. 3, 1422--1428.

\bibitem{Yoneama}
T. Yoneyama,  
On the $\frac{3}{2}$ stability theorem for one-dimensional delay-differential equations, 
J. Math. Anal. Appl. 125 (1987), 161--173.

\bibitem{Yoneyama1}
T. Yoneyama, 
The 3/2 stability theorem for one-dimensional delay-differential equations with unbounded
delay,” J. Math. Anal. Appl.  165  %no. 1, 
(1992), 133--143.

\bibitem{yskak} 
T. Yskak, 
 Stability of solutions of delay differential equations,
Siberian Adv. Math. 33 (2023), no. 3, 253--260.

\bibitem{Zhang}
B. Zhang, 
Contraction mapping and stability in a delay-differential equation, 
Dynamic Syst. Appl., 4 (2004), 183--190.

\bibitem{zhang1}
R. Zhang, D.  Zeng, J.H. Park, S. Zhong, Y. Liu, X.  Zhou, 
 New approaches to stability analysis for time-varying delay systems,
J. Franklin Inst. 356 (2019), no. 7, 4174--4189.
























\end{thebibliography}
\end{document}